 \documentclass[final,5p, twocolumn]{elsarticle}    

\usepackage{color}

\usepackage[utf8]{inputenc}
\usepackage{longtable}
\usepackage{amssymb}

\usepackage{flushend}

\usepackage{mathrsfs}
\usepackage{amsmath}

\usepackage{amsthm}

\newtheorem{thm}{Theorem}
\newtheorem{lem}[thm]{Lemma}

\newtheorem{prob}{Problem}

\newtheorem{defn}{Definition}

\theoremstyle{definition} 

\newtheorem{rem}{Remark}

\usepackage{cuted}
\usepackage{flushend}

\usepackage{tikz}
\usepackage{tkz-graph}
\usepackage{graphicx}

\usepackage{enumerate}

\usepackage{xcolor}
\usepackage{hyperref}

\hypersetup{
    colorlinks = true,
	linkcolor={red!80!black},
	citecolor={green!80!black},
	urlcolor={blue!80!black}
}

\begin{document}
	
	\begin{frontmatter}
		
		
		
		\title{{\bf $\mathcal{H}_2$ Suboptimal Output Synchronization of Heterogeneous Multi-Agent Systems}}

		\author{Junjie Jiao}\ead{j.jiao@rug.nl}
		\author{Harry L. Trentelman}\ead{h.l.trentelman@rug.nl}
		\author{M. Kanat Camlibel}\ead{m.k.camlibel@rug.nl}
		
		\address{Bernoulli Institute for Mathematics, Computer Science and Artificial Intelligence, University of Groningen, 9700 AK Groningen, The Netherlands}
		

		\begin{abstract}
This paper deals with the $\mathcal{H}_2$ suboptimal output synchronization problem for heterogeneous linear multi-agent systems. Given a multi-agent system with possibly distinct agents and an associated $\mathcal{H}_2$ cost functional, the aim is to design output feedback based protocols that guarantee the associated cost to be smaller than a given upper bound while the controlled network achieves output synchronization. 
A design method is provided to compute such protocols.
For each agent, the computation of its two local control gains involves two Riccati inequalities, each of dimension equal to the state space dimension of the agent. 
%
A simulation example is provided to illustrate the performance of the proposed protocols.
		\end{abstract}

	\end{frontmatter}
	

\section{Introduction}\label{sec_intro}
Over the last two decades, the problems of designing protocols   that  achieve consensus or synchronization in multi-agent systems have attracted much attention in the field of systems and control, see e.g. \cite{Olfati-Saber2004}, \cite{tamas2008}, \cite{zhongkui_li_unified_2010} and \cite{6303906}.  
The essential feature of these problems is that, while each agent makes use of only local state or output information to implement its  own local controller, the resulting  global protocol will   achieve consensus or synchronization for the global controlled multi-agent network \cite{Scardovi2009},  \cite{harry_2013}.
One of the challenging problems in this context is the problem of designing protocols that minimize given quadratic cost criteria while achieving consensus or synchronization, see e.g. \cite{Jiao2018}, \cite{Jiao2020H2output}, \cite{wei_ren2010}, \cite{kristian2014} and \cite{Nguyen2015}.
Due to the structural constraints imposed on the protocols, such optimal control problems are non-convex  and very difficult to solve. It is also unclear whether  in general    closed form solutions exist.

In the past, many efforts have been devoted to designing distributed protocols for {\em homogeneous} multi-agent systems that guarantee suboptimal or optimal performance and achieve {\em state} synchronization or consensus. 
In \cite{wei_ren2010}, this was done  for distributed linear quadratic control of multi-agent systems with {\em single integrator} agent dynamics, see also \cite{Jiao2019local}. 
In \cite{Nguyen2015} and \cite{Jiao2018},   multi-agent systems with general agent dynamics and a global linear quadratic cost functional were considered. 
In \cite{kristian2014} and \cite{huaguang_zhang2015}, an inverse optimal approach was adopted to address the distributed linear quadratic control problem, see also \cite{nguyen_2017}. 
For $\mathcal{H}_2$ cost functionals of a particular form,  \cite{LI2011797} and \cite{Zhongkui2014} proposed distributed suboptimal protocols that stabilize the controlled multi-agent network. 
In \cite{JIAO2018154}, a distributed $\mathcal{H}_2$ suboptimal control problem was addressed using static state feedback.  The results in \cite{JIAO2018154} were then generalized in \cite{Jiao2020H2output} to the case of dynamic output feedback.

More recently,    {\em output} synchronization problems for {\em heterogeneous} multi-agent systems have also attracted much attention. 
In \cite{WIELAND20111068}, it was shown that solvability of certain regulator equations is a necessary condition for output synchronization of heterogeneous multi-agent systems, and suitable protocols were proposed,  see also \cite{GRIP20122444}.
In \cite{Shim2011}, by embedding an internal model in the local controller of each agent, 
dynamic output feedback based protocols were proposed for a class of heterogeneous uncertain   multi-agent systems.
In \cite{Lunze2012}, it was shown that the outputs of the agents can be synchronized by a networked protocol if and only if these agents have certain dynamics in common. 
Later on, in \cite{Lunze2013} a linear quadratic control method was adopted for computing output synchronizing  protocols.
In  \cite{JIAO2016361}, an $\mathcal{L}_2$-gain output synchronization problem was addressed by casting this problem into a number of $\mathcal{L}_2$-gain stabilization problems for  certain linear systems, where the state space dimensions of these systems are equal to that of the agents.
For related work, we also mention \cite{Stephan2020},   \cite{ZHANG20142515} and \cite{SEYBOTH2015392}, to name a few.

Up to now, little attention has been paid in the literature to problems of designing output synchronizing protocols for heterogeneous multi-agent systems that guarantee  a certain performance.
In the present paper, we will deal with the problem of $\mathcal{H}_2$ optimal output synchronization for  heterogeneous linear multi-agent systems, i.e. the problem of minimizing a given $\mathcal{H}_2$ cost functional over all protocols that achieve output synchronization. 
%
%
Instead of addressing this {\em optimal} control problem, we will address a version of this problem that requires {\em suboptimality}.
More specifically, we will extend   previous results in \cite{Jiao2020H2output}   for  homogeneous multi-agent systems  to the case of  heterogeneous  multi-agent systems.
%

The outline of this paper is  as follows. 
In Section \ref{sec_preliminary}, we    provide some notation and graph theory used throughout this paper. 
In Section \ref{sec_problem}, we   formulate the   $\mathcal{H}_2$ suboptimal output synchronization problem.
In order to solve this problem, in Section \ref{sec_pre_results} we    review some basic material on $\mathcal{H}_2$ suboptimal control  by dynamic output feedback for linear systems, and some relevant results on   output synchronization  of heterogeneous multi-agent systems.
In Section \ref{sec_solution_output}, we    solve the problem introduced in Section \ref{sec_problem} and provide a design method for obtaining   $\mathcal{H}_2$  suboptimal protocols. 
%
%
 To illustrate the performance of our proposed protocols, a simulation example is provided in Section \ref{sec_simulation}.
Finally,   Section \ref{sec_conclusion}  concludes this paper.


\section{Notation and graph theory}\label{sec_preliminary}
%
\subsection{Notation}\label{subsec_notations}
We denote by $\mathbb{R}$  the field of real numbers  and by $\mathbb{C}$ the field of complex numbers.
%
%
The space of $n$ dimensional real vectors is denoted by $\mathbb{R}^n$.
We denote by $\mathbf{1}_n\in \mathbb{R}^n$ the vector with all its entries equal to $1$. 
For a  symmetric matrix $P$, we denote $P>0$ if $P$ is positive definite and $P < 0$  if $P$ is negative definite.
The identity matrix of dimension $n\times n$ is denoted by $I_n$.
The trace of a square matrix $A$ is denoted by ${\rm tr} (A)$.
A matrix is called Hurwitz if all its eigenvalues have negative real parts. 
We denote by $\text{diag}(d_1, d_2, \ldots, d_n)$ the $n \times n$ diagonal matrix with $d_1, d_2,\ldots, d_n$ on the diagonal. For given matrices $M_1,M_2, \ldots, M_n$, we denote by ${\rm blockdiag}(M_1,M_2, \ldots, M_n)$ the block diagonal matrix with diagonal blocks $M_i$.
%
%
The Kronecker product of two matrices $A$ and $B$ is denoted by $A \otimes B$.

\subsection{Graph theory}\label{subsec_graph}
A directed weighted graph is a triple $\mathcal{G} = (\mathcal{V}, \mathcal{E},\mathcal{A})$, where  $\mathcal{V} = \{ 1, 2,\ldots, N \}$ is the finite nonempty node set and   $\mathcal{E} = \{ e_1, e_2,\ldots, e_M \}$ with $\mathcal{E} \subset \mathcal{V} \times \mathcal{V}$ is the edge set, and $\mathcal{A} = [a_{ij}]$ is the adjacency matrix with nonnegative elements $a_{ij}$, called the edge weights. 
The entry $a_{ji}$ is nonzero if and only if  $(i,j) \in \mathcal{E}$.
A graph  is called simple if $a_{ii} =0$ for all $i$. It is called undirected if $a_{ij}= a_{ji}$ for all $i,j$.
Given a graph $\mathcal{G}$, a   path from node $1$ to node $p$ is a sequence of edges $(k, {k+1})$, $k = 1,2,\ldots, p-1$.  
A simple undirected graph is called connected if for each pair of nodes $i$ and $j$  there exists a   path from $i$ to $j$.
Given a simple undirected weighted graph $\mathcal{G}$, the degree matrix of $\mathcal{G}$ is defined by $\mathcal{D} = \textnormal{diag} (\delta_1,\delta_2,\ldots, \delta_N )$ with $\delta_{i} = \sum_{j=1}^{N} a_{ij}$. The Laplacian matrix is defined as ${L} := \mathcal{D} - \mathcal{A}$.
%
The Laplacian matrix of an undirected graph is symmetric and has only real nonnegative eigenvalues. 
A simple undirected weighted graph is connected if and only if  its Laplacian matrix ${L}$ has a simple eigenvalue at $0$. In that case there exists an orthogonal matrix $U$ such that $U^{\top} {L} U = \Lambda = \text{diag}(0, \lambda_2, \ldots, \lambda_N)$ with $0  < \lambda_2 \leq \cdots \leq \lambda_N$.
Throughout this paper it will be a standing assumption that the communication between the agents of the network is represented by a connected, simple undirected weighted graph.

A simple undirected weighted graph contains an even number of edges $M$. Define $K:= \frac{1}{2}M$. For such graph, an associated incidence matrix $R \in \mathbb{R}^{N \times K}$ is defined as a matrix $R = (r_1, r_2, \ldots, r_K)$ with columns $r_k \in \mathbb{R}^{N}$. Each column $r_k$ corresponds to exactly one pair of edges $e_k = \{(i,j), (j,i)\}$, and the  $i$th and $j$th entry of $r_k$ are equal to $1$ or $-1$, while they do not take the same value. The remaining entries of $r_k$ are equal to 0. 
%
We also define the matrix
\begin{equation}\label{W}
	W = \text{diag} ( \textsf{w}_1, \textsf{w}_2,\ldots, \textsf{w}_K)
\end{equation}
as the $K \times K$ diagonal matrix, where $w_k$ is the weight on each of the edges in $e_k$ for $k = 1,2,\ldots,K$.
The relation between the Laplacian matrix and the incidence matrix is captured by ${L} = R W R^{\top}$ \cite{6767074}. 

\section{Problem formulation}\label{sec_problem}
%
In this paper, we consider a heterogeneous  linear multi-agent system consisting of $N$ possibly distinct  agents.
The dynamics of the $i$th agent is represented by the linear time-invariant system
\begin{equation}\label{agent_output}
	\begin{aligned} 
		\dot{x}_i & = A_i x_i  + B_i u_i  + E_i d_i,\\
		y_i & = C_{1i} x_i  + D_{1i} d_i, \\
		z_i & = C_{2i} x_i  + D_{2i} u_i  ,
	\end{aligned}\qquad i = 1,2,\ldots,N,
\end{equation}
where $x_i \in \mathbb{R}^{n_i}$ is the state, $u_i \in \mathbb{R}^{m_i}$ is the coupling input, $d_i \in \mathbb{R}^{q_i}$ is an unknown external disturbance input, $y_i \in \mathbb{R}^{r_i}$ is the measured output and $z_i \in \mathbb{R}^p$ is the output to be synchronized.
The matrices $A_i$, $B_i$, $C_{1i}$, $D_{1i}$, $C_{2i}$, $D_{2i}$ and $E_i$ are of suitable dimensions.
Throughout this paper we assume that the pairs $(A_i, B_i)$ are stabilizable and the pairs $(C_{1i}, A_i)$ are detectable.
Since in \eqref{agent_output} the agents may have non-identical dynamics, in particular the state space dimensions of the agents  may differ. Therefore, one can not expect to achieve {\em state} synchronization for the network. 
Instead, in the context of heterogeneous networks  it is natural to consider {\em output} synchronization, see e.g. \cite{WIELAND20111068}, \cite{GRIP20122444}  and \cite{Lunze2012}.

It was shown in \cite{WIELAND20111068}  that  solvability of certain {\em regulator equations}  is necessary for output synchronization of heterogeneous linear multi-agent systems, see also \cite{GRIP20122444}, \cite{JIAO2016361}, \cite{SEYBOTH2015392} and \cite{BALDI2020}. Following up on this,  throughout this paper we make the  standard standing assumption that there exists a positive integer $r$ such that the regulator equations
\begin{equation}\label{regulation_eq}
\begin{aligned}
&A_i \Pi_i + B_i \Gamma_i = \Pi_i S , \\
&C_{2i} \Pi_i + D_{2i} \Gamma_i = R , \quad i= 1,2,\ldots,N
\end{aligned}
\end{equation}
have solutions $\Pi_i \in \mathbb{R}^{n_i\times r}$, $\Gamma_i\in \mathbb{R}^{m_i\times r}$, $R\in \mathbb{R}^{p\times r}$ and $S \in \mathbb{R}^{r\times r}$, where the eigenvalues of $S$ lie on the imaginary axis and the pair $(R, S)$ is observable.

%
Following \cite{WIELAND20111068}, we assume that the agents \eqref{agent_output} should be interconnected by a protocol of the form
\begin{equation}\label{protocol}
\begin{aligned}
\dot{w}_i & =  A_i w_i +B_i u_i + G_i(y_i -C_{1i} w_i),\\
\dot{v}_i &= S v_i +	\sum_{i=1}^N a_{ij} (v_j -v_i),\\
u_i & = F_i (w_i - \Pi_i v_i) + \Gamma_i v_i, \quad i= 1,2,\ldots,N,
\end{aligned}
\end{equation}
where $v_i \in \mathbb{R}^{r}$ and $w_i \in \mathbb{R}^{n_i}$ are the  states of the $i$th local controller, the matrices $S$, $\Pi_i$ and $\Gamma_i$ are solutions of \eqref{regulation_eq}, and  the matrices $F_i\in \mathbb{R}^{m_i \times n_i}$  and $G_i\in \mathbb{R}^{n_i \times r_i}$ are  control gains to be designed.
The coefficients $a_{ij}$ are the entries of the adjacency matrix $\mathcal{A}$ of the communication graph.
%
	We briefly explain the structure of this protocol.
	The first equation in \eqref{protocol} has the structure of  an asymptotic observer for the state of the $i$th agent.
	The second equation represents an  auxiliary system associated with the $i$th agent. 
	Each auxiliary system  receives the relative state values with respect to its neighboring auxiliary systems.
	In this way, the network of auxiliary systems will reach state synchronization.
	The third equation  in \eqref{protocol} is a static gain, it feeds back the value $w_i - \Pi_i v_i$  and the state $v_i$ of the associated auxiliary system to the $i$th agent.
	The idea of the protocol \eqref{protocol} is that, as time goes to infinity, the  state $x_i$  of the $i$th agent and its estimate $w_i$  converge to $\Pi_i v_i$ due the first equation in \eqref{regulation_eq}. 
	Subsequently, as a consequence of the second equation in \eqref{regulation_eq}, the outputs $z_i$ of the agents will  reach synchronization.
	%

Denote by  $\textbf{x}   = (x_1^{\top}, x_2^{\top}, \ldots,x_N^{\top})^{\top}$ the aggregate state vector  and likewise define $\textbf{u}$, $\textbf{v}$, $\textbf{w}$,  $\textbf{y}$, $\textbf{z}$ and $\textbf{d}$.
Denote by $A$ the block diagonal matrix 
\begin{equation}\label{matrixA}
A = \textnormal{blockdiag}(A_1, A_2,\ldots,A_N)
\end{equation} 
and likewise define $B$, $C_1$, $C_2$, $D_1$, $D_2$ and $E$.
The multi-agent system \eqref{agent_output} can then be written in compact form as
\begin{equation}\label{agent_compact_output}
	\begin{aligned}
		\dot{\textbf{x}} & = A \textbf{x} + B  \textbf{u} + E \textbf{d}, \\
		\textbf{y} & = C_1 \textbf{x} + D_1 \textbf{d},\\
		\textbf{z} &= C_2 \textbf{x} + D_2 \textbf{u}.
	\end{aligned}
\end{equation} 
Similarly, denote 
\[
F= \textnormal{blockdiag}(F_1, F_2,\ldots, F_N)
\]
and likewise define $G$, $\Gamma$ and $\Pi$. 
The protocol \eqref{protocol} can  be written in compact form as 
\begin{equation}\label{protocol_compact}
\begin{aligned}
\dot{\textbf{w}}  & =  A \textbf{w}  + B  \textbf{u} + G (\textbf{y} - C_1 \textbf{w}),\\
\dot{\textbf{v}} &= (I_N \otimes S -L \otimes I_r) \textbf{v},\\
\textbf{u} &= F \textbf{w} + (\Gamma - F \Pi) \textbf{v}.
\end{aligned}
\end{equation}
Next, denote 
\[\textbf{x}_o = 
(\textbf{x}^\top,  \textbf{w}^\top, \textbf{v}^\top)^\top.
\]
By interconnecting the system \eqref{agent_compact_output} and the protocol \eqref{protocol_compact}, the controlled network is then represented  in compact form by
\begin{equation}\label{mas_output}
	\begin{aligned} 
		& \dot{\textbf{x}}_o
		=
		A_o
		\textbf{x}_o +
		E_o \textbf{d}, \\
		&	\textbf{z} = 
		C_o
		\textbf{x}_o,
	\end{aligned}
\end{equation}
where 
\begin{align*}
	A_o & = 
	\begin{pmatrix}
		A & B F & B \Gamma -B F  \Pi  \\
		G  C_1  &  A + B F - G  C_1 &  B \Gamma - B F \Pi \\
		0 & 0 & I_N\otimes S - L \otimes I_r
	\end{pmatrix}, \\
	C_o &= 
	\begin{pmatrix}
		C_{2}   &  D_{2} F & 
		D_{2} \Gamma - D_{2} F \Pi 
	\end{pmatrix}, \quad 
	E_o  = 
	\begin{pmatrix}
		E  \\
		G D_{1} \\
		0
	\end{pmatrix}.
\end{align*}
Foremost, we want the protocol \eqref{protocol} to achieve output synchronization for the overall network:
\begin{defn}\label{defn1}
	The protocol \eqref{protocol} is said to achieve $\textbf{z}$-output synchronization for the network \eqref{mas_output} if, for all $i, j = 1, 2, \ldots, N$, we have $z_i(t) - z_j(t) \to 0$, $v_i(t) - v_j(t) \to 0$ and $w_i(t) - w_j(t) \to 0$ as $t \to \infty$.
\end{defn}

%
In the context of output synchronization, we are interested in the differences of the output values of the agents in the controlled network.
Since the differences of the output values of communicating agents are captured by the incidence matrix $R$ of the communication graph \cite{Mesbahi2010}, we  define a performance output variable as 
\begin{equation*}
	\boldsymbol{\zeta} = (W^{\frac{1}{2}} R^{\top} \otimes I_p )\textbf{z},
\end{equation*}
where $W$ is the weight matrix defined in \eqref{W}. 
The output $\boldsymbol{\zeta}$   reflects the weighted disagreement between the outputs of the agents in accordance with the weights of the edges connecting these agents.
Subsequently, we have the following equations for the controlled network
\begin{equation}\label{network_output}
	\begin{aligned} 
		&	\dot{\textbf{x}}_o
		=
		A_o
		\textbf{x}_o +
		E_o \textbf{d}, \\
		&	\textbf{z}  = 
		C_o
		\textbf{x}_o,\\
		&	\boldsymbol{\zeta}  = 
		{C}_p
		\textbf{x}_o,
	\end{aligned}
\end{equation}
where 
\[
{C}_p = (W^{\frac{1}{2}} R^{\top} \otimes I_p ) C_o.
\]
The impulse response matrix of the disturbance $\textbf{d}$ to the performance output $\boldsymbol{\zeta}$ is   given by
\begin{equation}\label{inpulse_response}
	T_d(t) =  {C}_p e^{A_o t}E_o.
\end{equation}
The performance of the network is now quantified by the $\mathcal{H}_2$-norm of this impulse response. Thus we define the associated $\mathcal{H}_2$ cost functional as
\begin{equation}\label{h2_cost_output}
	J := \int_{0}^{\infty} \textnormal{tr}\left[ T_d^{\top}(t) T_d(t)\right] dt.
\end{equation}
Note that the cost functional \eqref{h2_cost_output} is a function of the gain matrices  $F_1,F_2,\ldots, F_N$ and $G_1,G_2,\ldots, G_N$.

The $\mathcal{H}_2$ optimal output synchronization problem is now defined as the problem of minimizing the cost functional \eqref{h2_cost_output} over all protocols \eqref{protocol} that achieve output synchronization.
Since  the protocol \eqref{protocol} has a particular structure imposed by the communication topology, the  $\mathcal{H}_2$ optimal output synchronization problem is a non-convex optimization problem, and it is  unclear whether a closed form solution exists in general. 
Therefore,  in this paper we will address a version of this problem that only requires {\em suboptimality}. 
The aim of this paper is then to  design a  protocol of the form \eqref{protocol} that guarantees the  associated cost  \eqref{h2_cost_output}  to be smaller than an a priori given upper bound while achieving $\textbf{z}$-output synchronization for  the network.  
More concretely, the problem we will address is the following:
\begin{prob}\label{prob}
	%
	Let $\gamma > 0$ be a given tolerance. Design gain matrices $F_1, F_2, \ldots, F_N$ and $G_1, G_2, \ldots, G_N$ such that the resulting protocol \eqref{protocol} achieves $\textbf{z}$-output synchronization and its associated cost \eqref{h2_cost_output} satisfies $J < \gamma$.
\end{prob}

To solve Problem \ref{prob}, in the next section we will first review some preliminary results on $\mathcal{H}_2$ suboptimal control for linear systems and on output synchronization of heterogeneous linear multi-agent systems. It will become clear later on that these preliminary results are necessary ingredients to address Problem \ref{prob}. 

\section{Preliminary results}\label{sec_pre_results}
 
\subsection{$\mathcal{H}_2$ suboptimal control  for linear systems by dynamic output feedback }\label{sec_output}
In this subsection, we will review the  $\mathcal{H}_2$ suboptimal control problem by dynamic output feedback for linear systems, see e.g. \cite{Scherer2000}, \cite{Scherer1997}, \cite{algebraic_1997}, \cite{HAESAERT2018306} and \cite{Jiao2020H2output}. In particular, we will review the results  from \cite{Jiao2020H2output} on  separation principle based  $\mathcal{H}_2$   suboptimal control for continuous-time linear systems.

Consider the  system
\begin{equation}\label{sys_xyz}
	\begin{aligned} 
		\dot{x} & = \bar{A} x + \bar{B} u + \bar{E} d,\\
		y &= \bar{C}_1 x + \bar{D}_1 d, \\
		z &= \bar{C}_2 x + \bar{D}_2 u,
	\end{aligned}
\end{equation}
where  $x \in \mathbb{R}^n$ is the state, $u \in \mathbb{R}^m$ is the control input, $d\in \mathbb{R}^q$ is an unknown external disturbance input, $y \in \mathbb{R}^r$ is the measured output, and $z \in \mathbb{R}^p$ is the output to be controlled. 
The matrices $\bar{A}$, $\bar{B}$, $\bar{C}_1$, $\bar{C}_2$, $\bar{D}_1$, $\bar{D}_2$ and $\bar{E}$  are of suitable dimensions. 
We assume that the pair $(\bar{A}, \bar{B})$ is stabilizable and the pair $(\bar{C}_1, \bar{A})$ is detectable.
We consider  dynamic output feedback controllers of the form
\begin{equation}\label{dyna_w}
	\begin{aligned}
		\dot{w} &= \bar{A} w + \bar{B} u + G \left(y -\bar{C}_1 w\right), \\
		u &= F w,
	\end{aligned}
\end{equation}
where $w \in \mathbb{R}^n$ is the state of the controller,   $F \in \mathbb{R}^{m \times n}$ and $G \in \mathbb{R}^{n \times r}$ are gain matrices to be designed.
By interconnecting the controller \eqref{dyna_w} and the system \eqref{sys_xyz}, we obtain the controlled system 
\begin{equation}\label{sys_dyna_w}
	\begin{aligned}
		\begin{pmatrix}
			\dot{x}\\
			\dot{w}
		\end{pmatrix} 
		& =
		\begin{pmatrix}
			\bar{A} & \bar{B}F \\
			G \bar{C}_1 & \bar{A} +  \bar{B}F -G \bar{C}_1  
		\end{pmatrix} 
		\begin{pmatrix}
			x \\
			w
		\end{pmatrix}
		+
		\begin{pmatrix}
			\bar{E} \\
			G \bar{D}_1	
		\end{pmatrix}d ,
		\\
		z & =
		\begin{pmatrix}
			\bar{C}_2 & \bar{D}_2 F 
		\end{pmatrix}
		\begin{pmatrix}
			x \\
			w
		\end{pmatrix}.
	\end{aligned} 
\end{equation}
Denote 
$	{A}_e  =  	
	\begin{pmatrix}
		\bar{A} & \bar{B}F \\
		G \bar{C}_1 & \bar{A} + \bar{B}F -G \bar{C}_1 
	\end{pmatrix} ,$
	$	{E}_e   =
	\begin{pmatrix}
	\bar{E} \\
	G \bar{D}_1	
	\end{pmatrix},$
	${C}_e  = 
	\begin{pmatrix}
		\bar{C}_2 & \bar{D}_2 F 
	\end{pmatrix}$.
The impulse response matrix  of the disturbance $d$ to the output $z$ is given by ${T}_{F,G}(t) = {C}_e e^{{A}_e t}{E}_e$.
We define the $\mathcal{H}_2$ cost functional as
\begin{equation}\label{cost_dyna}
	J(F,G)  := \int_{0}^{\infty} 
	\text{tr} \left[ {T}_{F,G}^{\top}(t) {T}_{F,G}(t) \right] dt.
\end{equation}
The  $\mathcal{H}_2$ suboptimal control problem by dynamic output feedback is the problem of finding a controller of the form \eqref{dyna_w} such that the associated cost \eqref{cost_dyna} is smaller than an a priori given upper bound and the controlled system \eqref{sys_dyna_w} is internally stable. 
The following lemma provides a design method for computing such a controller, see also \cite[Theorem 4]{Jiao2020H2output}.
\begin{lem}\label{linear_h2}
	%
	Let $\gamma > 0$ be a given tolerance.
	%
	%
	Assume that $\bar{D}_1 \bar{E}^{\top} =0$, $\bar{D}_2^{\top} \bar{C}_2 =0$ and  $\bar{D}_1  \bar{D}_1^{\top} =I_r$,  $\bar{D}_2^\top \bar{D}_2 = I_m$. 
Let $P >0$ and $Q>0$ satisfy the Riccati inequalities
	\begin{align*}
	\bar{A}^\top P + P \bar{A} -P \bar{B}  \bar{B}^\top P + \bar{C}_2^\top \bar{C}_2 & <0,  \\
	\bar{A} Q + Q \bar{A}^\top - Q \bar{C}_1^\top  \bar{C}_1 Q + \bar{E} \bar{E}^\top  & <0  .
	\end{align*}
	If, in addition, such $P$ and $Q$ satisfy
\begin{equation*}
	\textnormal{tr} \left(\bar{C}_1 Q  P Q \bar{C}_1^\top \right) + \textnormal{tr} \left( \bar{C}_2 Q  \bar{C}_2^\top  \right) < \gamma,
\end{equation*}
	then the controller \eqref{dyna_w} with $F = - \bar{B}^\top P$ and $G = Q \bar{C}_1^\top$  internally stabilizes the system \eqref{sys_xyz} and   is suboptimal, i.e. $J(F, G) <\gamma$.
\end{lem}
For a proof of Lemma \ref{linear_h2}, we refer to  \cite[Theorem 4]{Jiao2020H2output}.

\subsection{Output synchronization of heterogeneous linear  multi-agent systems}\label{out_synch}
In this subsection, we will review some relevant results on  output synchronization of heterogeneous linear multi-agent systems, see  also \cite{WIELAND20111068}, \cite{GRIP20122444},  \cite{Shim2011} and \cite{Lunze2012}.

Consider a heterogeneous  linear multi-agent system consisting of $N$ possibly distinct  agents. The dynamics of the $i$th agent is represented by the linear time-invariant system
\begin{equation}\label{agent_pre}
	\begin{aligned}  
		\dot{x}_i  &= A_i x_i   + B_i u_i, \\
			y_i & = C_{1i} x_i ,\\
		z_i &= C_{2i} x_i + D_{2i} u_i,
	\end{aligned}\quad i = 1,2,\ldots,N.
\end{equation}
The agents \eqref{agent_pre} will be   interconnected by  	a protocol  of the form \eqref{protocol}, where the matrices $S$, $\Gamma_i$ and $\Pi_i$ are assumed to satisfy the regulator equations \eqref{regulation_eq}.
The multi-agent system \eqref{agent_pre} can  be written in compact form as
\begin{equation}\label{agent_compact}
\begin{aligned}
\dot{\textbf{x}} & = A \textbf{x} + B  \textbf{u} , \\
\textbf{y} & = C_1 \textbf{x},\\
\textbf{z} &= C_2 \textbf{x} + D_2 \textbf{u},
\end{aligned}
\end{equation} 
and the protocol \eqref{protocol} can  be written as \eqref{protocol_compact}.
By interconnecting the system \eqref{agent_compact} and the protocol \eqref{protocol_compact}, the controlled network is then given by
\begin{equation}\label{mas}
\begin{aligned} 
& \dot{\textbf{x}}_o
=
A_o
\textbf{x}_o, \\
&	\textbf{z} = 
C_o
\textbf{x}_o.
\end{aligned}
\end{equation}

The following lemma yields conditions under which the controlled network \eqref{mas} achieves $\textbf{z}$-output synchronization. 
\begin{lem}\label{lem_synchronization}
	Consider the multi-agent system \eqref{agent_pre} and the protocol \eqref{protocol}.
	%
	%
	Let gain matrices $F_i$ and $G_i$ be such that the matrices $A_i + B_i F_i$ and $A_i - G_i C_{1i}$ are Hurwitz.
	Then the associated protocol \eqref{protocol} achieves  $\textbf{z}$-output synchronization for the network.
\end{lem}
A proof of Lemma \ref{lem_synchronization} can be given along the lines of the proof of \cite[Theorem 5]{WIELAND20111068}.

We are now ready to deal with the $\mathcal{H}_2$ suboptimal output synchronization problem formulated in Problem \ref{prob}.

\section{Design of $\mathcal{H}_2$ suboptimal output synchronization protocols using dynamic output feedback} \label{sec_solution_output}
In this section, we will resolve Problem \ref{prob}.
More specifically, we will establish a design method for computing gain matrices  $F_1, F_2, \ldots, F_N$ and $G_1, G_2, \ldots, G_N$ such that the associated protocol \eqref{protocol}  achieves $\textbf{z}$-output synchronization and guarantees $J  < \gamma$.

In the sequel, we will first show that this problem  can be simplified by transforming it into $\mathcal{H}_2$ suboptimal  control problems for $N$ auxiliary systems.
The suboptimal gains $F_i$ and $G_i$ for these $N$ separate problems will turn out to also yield a suboptimal protocol for the heterogeneous network.

To this end, we introduce  the following $N$ auxiliary systems
\begin{equation}\label{sub_system_output}
	\begin{aligned}
		\dot{\xi}_i  & = {A}_i \xi_i + {B}_i \nu_i  + {E}_i \delta_i, \\
		\vartheta_i &= C_{1i} \xi_i + D_{1i} \delta_i, \\
		\eta_i & = {C}_{2i} \xi_i + {D}_{2i} \nu_i,\quad i = 1,2,\ldots, N,
	\end{aligned}
\end{equation}
where $\xi_i \in \mathbb{R}^{n_i}$ is the state, $\nu_i\in \mathbb{R}^{m_i}$ is the coupling input, $\delta_i \in \mathbb{R}^{q_i}$ is an unknown external disturbance input, $\vartheta_i \in \mathbb{R}^{r_i}$ is the measured output and $\eta_i \in \mathbb{R}^p$ is the output to be controlled.
For given gain matrices $F_i$ and $G_i$, consider the dynamic output feedback controllers
\begin{equation}\label{sub_controller_output}
\begin{aligned}
\dot{\omega}_i  & = {A}_i \omega_i + {B}_i \nu_i  + {G}_i( \vartheta_i - C_{1i} \omega_i), \\
\nu_i & = F_i \omega_i, \quad i = 1,2,\ldots, N,
\end{aligned}
\end{equation}
where $ \omega_i \in \mathbb{R}^n$ is the state of the $i$th controller.

By interconnecting the systems \eqref{sub_system_output} and the controllers \eqref{sub_controller_output}, we obtain the $N$ controlled auxiliary systems
\begin{equation}\label{sub_system_closed}
\begin{aligned}
\begin{pmatrix}
\dot{\xi}_i \\ \dot{\omega}_i
\end{pmatrix}
& = 
\begin{pmatrix}
{A}_i & {B}_i {F}_i \\
{G}_i{C}_{1i} &  {A}_i + {B}_i {F}_i - {G}_i{C}_{1i}
\end{pmatrix}
\begin{pmatrix}
{\xi}_i \\ {\omega}_i
\end{pmatrix} 
+ 
\begin{pmatrix}
 {E}_i \\ {G}_i  D_{1i}
\end{pmatrix} \delta_i,\\
\eta_i & = 
\begin{pmatrix}
{C}_{2i} & {D}_{2i} {F}_i 
\end{pmatrix}
\begin{pmatrix}
{\xi}_i \\ {\omega}_i
\end{pmatrix} , \quad i = 1,2,\ldots, N.
\end{aligned}
\end{equation}
For $i = 1,2,\ldots, N$, denote 
\begin{align*}
	\bar{A}_i & = 
	\begin{pmatrix}
	{A}_i & {B}_i {F}_i \\
	{G}_i{C}_{1i} &  {A}_i + {B}_i {F}_i - {G}_i{C}_{1i}
	\end{pmatrix},\\
	\bar{C}_i & =  
	\begin{pmatrix}
	{C}_{2i} & {D}_{2i} {F}_i 
	\end{pmatrix},\quad 
	\bar{E}_i =  
	\begin{pmatrix}
	{E}_i \\ {G}_i  D_{1i}
	\end{pmatrix}.
\end{align*}
The impulse response matrix of the disturbance $\delta_{i}$ to the output $\eta_i$  is equal to
\begin{equation*}
	{T}_{\delta i}(t) =  \bar{C}_{i} e^{\bar{A}_i t} \bar{E}_i,
\end{equation*}
and an associated $\mathcal{H}_2$ cost functional is  defined as
\begin{equation}\label{cost_sub_sys_output}
	{J}_i = \int_{0}^{\infty} \textnormal{tr} [ {T}_{\delta i}^\top(t)  {T}_{\delta i}(t)] dt.
\end{equation}
The following lemma holds.
\begin{lem}\label{cost_trans}
	Let $\gamma >0$ be a given tolerance.
	Assume, for $i=1,2,\ldots, N$, the systems \eqref{sub_system_closed} are internally stable and the costs \eqref{cost_sub_sys_output} satisfy 
	\begin{equation}\label{cost_sum}
		 \sum_{i=1}^N {J}_i < \frac{\gamma}{\lambda_N},
	\end{equation}
	where $\lambda_N$ is the largest eigenvalue of the Laplacian matrix $L$.
	Then the protocol \eqref{protocol} achieves  $\textbf{z}$-output synchronization for the network \eqref{network_output} and the associated cost \eqref{h2_cost_output} satisfies ${J} < \gamma$.
\end{lem}
\begin{proof}
First, note that the systems \eqref{sub_system_closed} are internally stable if and only if the matrices $A_i + B_i F_i$ and $A_i - G_i C_{1i}$ are Hurwitz, see e.g. \cite[Section 3.12]{harry_book}. 
Hence, by Lemma \ref{lem_synchronization}, if the systems \eqref{sub_system_closed} are internally stable, then the  network controlled using the protocol \eqref{protocol} reaches  $\textbf{z}$-output synchronization.

Next, we will show that if \eqref{cost_sum} holds, then $J < \gamma$. Note that \eqref{cost_sum} is equivalent to
\begin{equation}\label{cost_midstep1}
	\lambda_N \sum_{i=1}^N  \int_{0}^{\infty} \textnormal{tr} [ {T}_{\delta i}^\top (t)  {T}_{\delta i}(t)] dt < \gamma.
\end{equation}
In turn, the inequality \eqref{cost_midstep1} holds if and only if 
\begin{equation}\label{cost_midstep2}
\lambda_N   \int_{0}^{\infty} \textnormal{tr} [ \bar{T}_d^\top(t)  \bar{T}_d(t)] dt < \gamma
\end{equation}
holds, where 
\[\bar{T}_d = \bar{C}_o e^{\bar{A}_o t}\bar{E}_o
\]
with
\begin{align*}
\bar{A}_o &=  
\begin{pmatrix}
	A & B F   \\
G  C_1  &  A + B F - G  C_1
\end{pmatrix},\quad
\bar{E}_o = 
\begin{pmatrix}
E  \\
G D_{1}
\end{pmatrix},\\
\bar{C}_o &=  
	\begin{pmatrix}
C_{2}   &  D_{2} F
\end{pmatrix}.
\end{align*}
Recall that the matrix $A$  is the block diagonal matrix defined in 
\eqref{matrixA}, similarly for the matrices $B$,  $C_1$, $C_2$, $D_1$, $D_2$,  $E$,  $F$ and  $G$.
Using the fact that $\lambda_N I_{pN} - L\otimes I_p \geq 0$, it can be shown that \eqref{cost_midstep2} implies
\begin{equation}\label{cost_midstep3}
 \int_{0}^{\infty} \textnormal{tr} [ \bar{T}_d^\top (t)(L \otimes I_p)  \bar{T}_d(t)] dt < \gamma.
\end{equation}
On the other hand,   
\begin{equation}\label{cost_midstep4}
	  \int_{0}^{\infty} \textnormal{tr} [ \bar{T}_d^\top (t) (L \otimes I_p)  \bar{T}_d(t)] dt
	  =
	  \int_{0}^{\infty} \textnormal{tr}\left[ T_d^{\top}(t) T_d(t)\right] dt
\end{equation}
with $T_d(t)$ given by \eqref{inpulse_response}.
Note that the right hand side of \eqref{cost_midstep4} is exactly the  cost $J$ given by \eqref{h2_cost_output} associated with the  network \eqref{network_output}.
It  follows that $J < \gamma$. 
This completes the proof.
\end{proof}

By the previous, if the gain matrices $F_i$ and $G_i$ are such that $A_i +B_i F_i$ and $A_i - G_i C_{1i}$ are Hurwitz and \eqref{cost_sum} holds, then the protocol \eqref{protocol} using these $F_i$ and $G_i$ yields $\textbf{z}$-output synchronization and $J<
\gamma$.
In the next theorem, we will provide a  method for computing  gain matrices $F_i$ and $G_i$  such that the  above holds.
\begin{thm}\label{main_thm}
	Let $\gamma >0$ be a given tolerance. 
	For $i=1,2,\ldots, N$, assume that $D_{1i} E_i^\top =0$, $D_{2i}^\top C_{2i} =0$, $D_{1i} D_{1i}^\top = I_{r_i}$ and $D_{2i}^\top D_{2i} = I_{m_i}$. 
	Let $P_i >0$ satisfy 
	\begin{equation}\label{ineq_Pi}
		A_i^\top P_i + P_i A_i^\top - P_i B_i B_i^\top P_i +  C_{2i}^\top C_{2i}<0.
	\end{equation} 
	Let  $Q_i >0$ satisfy 
	\begin{equation}\label{ineq_Qi}
		A_i  Q_i + Q_i A_i^\top - Q_i C_{1i}^\top C_{1i} Q_i +  E_i E_i^\top <0.
	\end{equation} 
	If, in addition, such $P_i$ and $Q_i$ satisfy
	\begin{equation}\label{gamma}
		\textnormal{tr} (C_{1i} Q_i P_i Q_i C_{1i}^\top) +	\textnormal{tr} (C_{2i} Q_i C_{2i}^\top) < \frac{\gamma}{N\lambda_N},
	\end{equation}
	then the protocol \eqref{protocol} with $F_i: = - B_i^\top P_i$ and $G_i := Q_i C_{1i}^\top$  achieves $\textbf{z}$-output synchronization for the network \eqref{network_output} and guarantees $J < \gamma$.
\end{thm}
\begin{proof}
	Note that \eqref{ineq_Pi} is equivalent to
	\begin{align}
		(A_i - B_i B_i^\top P_i)^\top P_i + (A_i - B_i B_i^\top P_i) & \nonumber \\
		  + P_i B_i B_i^\top P_i +  C_{2i}^\top C_{2i} &<0 \label{Pi}
	\end{align}
	and  \eqref{ineq_Qi} is equivalent to
	\begin{align}
		 (A_i -  Q_i C_{1i}^\top C_{1i})  Q_i + Q_i (A_i-  Q_i C_{1i}^\top C_{1i} )^\top &\nonumber\\
	 + Q_i C_{1i}^\top C_{1i} Q_i +  E_i E_i^\top & <0. \label{Qi}
	\end{align}	
Taking $F_i := - B_i^\top P_i$ and $G_i := Q_i C_{1i}^\top$, it then follows that  $A_i + B_i F_i$ and $A_i - G_i C_{1i}$  are Hurwitz.

Next, by \eqref{gamma}, it follows from Lemma \ref{linear_h2} that 
\[J_i < \frac{\gamma}{N\lambda_N}, \quad i =1,2,\ldots,N.
\]
Thus we have \eqref{cost_sum}, and the conclusion then follows from  Lemma \ref{cost_trans}. 
\end{proof}
We note that the conditions $D_{1i} E_i^\top =0$, $D_{2i}^\top C_{2i} =0$, $D_{1i} D_{1i}^\top = I_{r_i}$ and $D_{2i}^\top D_{2i} = I_{m_i}$ are made here to simplify notation, and can be relaxed to the regularity conditions $D_{1i} D_{1i}^\top > 0$ and $D_{2i}^\top D_{2i} >0 $ alone. 
\begin{rem}
	In Theorem \ref{main_thm}, in order to select $\gamma$, the followings steps could be taken.
	For $i=1,2\ldots,N$:
	\begin{enumerate}[(i)]
		\item Compute positive definite solutions $P_i$  and $Q_i$ of the Riccati  inequalities \eqref{ineq_Pi} and  \eqref{ineq_Qi}. Such solutions exist.
		\item Denote $S_i = 	\textnormal{tr} (C_{1i} Q_i P_i Q_i C_{1i}^\top) +	\textnormal{tr} (C_{2i} Q_i C_{2i}^\top)$.
		\item Choose $\gamma$ such that $N\lambda_N S_i <\gamma$.
	\end{enumerate}
Note that the smaller $S_i$ or $\lambda_N$ is, the smaller such feasible $\gamma$ is allowed to be. Unfortunately, the problem of minimizing $S_i$ over all  $P_i>0$ and $Q_i>0$ that satisfy \eqref{ineq_Pi} and \eqref{ineq_Qi} is a non-convex optimization problem.
However, since smaller $Q_i$ leads to smaller $\textnormal{tr} (C_{2i} Q_i C_{2i}^\top)$  and smaller $P_i$ and  $Q_i$  lead to smaller 	$\textnormal{tr} (C_{1i} Q_i P_i Q_i C_{1i}^\top)$, and consequently smaller feasible $\gamma$, we could try to find $P_i$ and  $Q_i$ as small as possible.
In  fact, one can find $P_i = P_i(\epsilon_i)>0$ to \eqref{ineq_Pi} by solving the Riccati equation
	\begin{equation*}
A_i^\top P_i + P_i A_i^\top - P_i B_i B_i^\top P_i +  C_{2i}^\top C_{2i} +\epsilon_i I_{n_i} =0
\end{equation*} 
with $\epsilon_i>0$ arbitrary.
Similarly, one can find $Q_i = Q_i(\sigma_i) >0$ to \eqref{ineq_Qi}  by solving the dual Riccati equation
	\begin{equation*}
A_i  Q_i + Q_i A_i^\top - Q_i C_{1i}^\top C_{1i} Q_i +  E_i E_i^\top + \sigma_i I_{n_i} =0
\end{equation*} 
with $\sigma_i>0$ arbitrary. 
By using a standard argument, it can be shown that $P_i(\epsilon_i)$ and $Q_i(\sigma_i) $  decrease as $\epsilon_i$ and $\sigma_i$ decrease, respectively. So $\epsilon_i$ and $\sigma_i$  should be taken close to $0$ to get smaller $P_i$ and $Q_i$.
\end{rem}

\section{Simulation example}\label{sec_simulation}
In this section, we will give a simulation example based on the example in \cite{WIELAND20111068} to illustrate the design method of Theorem \ref{main_thm}.

Consider a network of $N=6$ heterogeneous agents. The dynamics of the agents are given by
\begin{equation*}
\begin{aligned} 
\dot{x}_i & = A_i x_i  + B_i u_i  + E_i d_i,\\
y_i & = C_{1i} x_i  + D_{1i} d_i, \\
z_i & = C_{2i} x_i  + D_{2i} u_i  ,
\end{aligned}\qquad i = 1,2,\ldots,6,
\end{equation*}
where
 $A_i=
	\begin{pmatrix}
	0 & 1 & 0 \\
	0 & 0 &c_i \\
	0 & -f_i & -a_i
	\end{pmatrix}$, 
	$B_i = 
	 \begin{pmatrix}
	 0 \\ 0 \\ b_i
	 \end{pmatrix},$
	$ E_i = 
	 \begin{pmatrix}
	 0 & 0.2 \\
	 0 & 0 \\
	 0 & 0.2
	 \end{pmatrix},$
	$ C_{1i}  =  	
	 \begin{pmatrix}
	 1 & 0 & 0
	 \end{pmatrix}, $
	 $D_{1i} = 
	 \begin{pmatrix}
	 1 & 0
	 \end{pmatrix},$
$ C_{2i}  = 
	 \begin{pmatrix}
	 1 & 1 & 0\\
	 0 & 0 & 0
	 \end{pmatrix},$
	$ D_{2i} =
	 \begin{pmatrix}
	 0 \\ 1
	 \end{pmatrix}$.
The parameters $a_i$, $b_i$, $c_i$ and $f_i$ are chosen to be
\begin{align*}
	& a_i = 2, \ c_i = 1, \quad i = 1,2,\ldots,6 ,\\
	& b_1 = b_4 = 1,\ b_2 = b_5 = 2,\ b_3 = b_6 =  3, \\
	& f_1 = f_4 = 1,\ f_2 = f_5 = 2, \ f_3 = f_6 = 3.
\end{align*}
The pairs $(A_i, B_i)$ are stabilizable and the pairs $(C_{1i}, A_{i})$ are detectable. 
We also have that $D_{1i} E_i^\top =0$, $D_{2i}^\top C_{2i} =0$, $D_{1i} D_{1i}^\top = 1$ and $D_{2i}^\top D_{2i} = 1$.
The communication graph between the six agents is assumed to be an  undirected cycle graph.
%
The largest eigenvalue of  the corresponding Laplacian matrix $L$ is $\lambda_6 = 4$.

We choose the matrices $S$ and $R$ in  the regulator equations \eqref{regulation_eq} to be
\begin{equation*}
S =
\begin{pmatrix}
	0 & 1 \\
	0 & 0
\end{pmatrix}, \quad
R =
\begin{pmatrix}
	1 & 1  \\
	0 & 1
\end{pmatrix}.
\end{equation*}
The eigenvalues of $S$ are on the imaginary axis and the pair $(R, S)$ is observable.
We  solve the   equations \eqref{regulation_eq} and compute
\begin{equation*}
	\Pi_i = 
	\begin{pmatrix}
		1 & 0 \\
		0 & 1 \\
		0 & 0
	\end{pmatrix},\quad
	\Gamma_i =
	\begin{pmatrix}
		0 & 1
	\end{pmatrix},\quad i = 1,2,\ldots, 6.
\end{equation*}
The objective is to design a protocol of the form \eqref{protocol} such that the associated cost \eqref{h2_cost_output} satisfies $J < \gamma$ while achieving $\textbf{z}$-output synchronization.  Let the desired upper bound be $\gamma = 18$.

Following  the design method in Theorem \ref{main_thm}, for $i = 1,2,
\ldots, 6$, we  compute a positive definite solution $P_i$ to \eqref{ineq_Pi} by solving the  Riccati equation
\begin{equation*}
		A_i^\top P_i + P_i A_i^\top - P_i B_i B_i^\top P_i +  C_{2i}^\top C_{2i} + \epsilon I_{n_i} = 0
\end{equation*}
with $\epsilon =0.001$.
We also compute a positive definite solution $Q_i$ to \eqref{ineq_Pi} by solving the dual  Riccati equation
\begin{equation*}
		A_i  Q_i + Q_i A_i^\top - Q_i C_{1i}^\top C_{1i} Q_i +  E_i E_i^\top  + \sigma I_{n_i} = 0
\end{equation*}
with $\sigma =0.001$.
Accordingly, we compute the associated gain matrices $F_i$ and $G_i$ to be
\begin{align*}
F_1= F_4 &= 
\begin{pmatrix}
	   -1.0005  &  -1.7329   & -0.7326
\end{pmatrix}, \\
F_2 = F_5 &= 
\begin{pmatrix}
   -1.0005 &  -1.2345   & -0.4951
\end{pmatrix}, \\
F_3 = F_6 &= 
\begin{pmatrix}
   -1.0005  & -1.0327 &  -0.3982
\end{pmatrix},
\end{align*}
and
\begin{align*}
G_1 = G_4 &= 
\begin{pmatrix}
    0.3290 &
0.0341 &
0.0028
\end{pmatrix}^\top,\\
G_2 =G_5 &= 
\begin{pmatrix}
    0.2804 &
0.0193 &
0.0007
\end{pmatrix}^\top,\\
G_3  = G_6 &= 
\begin{pmatrix}
    0.2578 &
0.0132 &
0.0002
\end{pmatrix}^\top.
\end{align*}

As an example, we take the initial states of the agents to be 
$x_{10} = 
\begin{pmatrix}
1.0 & 1.4 & 1.6
\end{pmatrix}^\top$,
$x_{20} = 
\begin{pmatrix}
1.2 & -1.7 & 0.5
\end{pmatrix}^\top$,
$x_{30} = 
\begin{pmatrix}
1.3 & -1.2 & 1.3
\end{pmatrix}^\top$,
$x_{40} = 
\begin{pmatrix}
0.6 & 1.6 & -1.3
\end{pmatrix}^\top$,
$x_{50} = 
\begin{pmatrix}
1.8 & 1.5 & 1.6
\end{pmatrix}^\top$,
$x_{60} = 
\begin{pmatrix}
-1.1 & 1.7 & 0.9
\end{pmatrix}^\top$.
We take the initial states $w_i$ to be zero, and the initial states $v_i$ to be 
$v_{10} = 
\begin{pmatrix}
0.9 & 1.1
\end{pmatrix}^\top$,
$v_{20} = 
\begin{pmatrix}
0.8 & 1.4
\end{pmatrix}^\top$,
$v_{30} = 
\begin{pmatrix}
-1.0 & 0.9
\end{pmatrix}^\top$,
$v_{40} = 
\begin{pmatrix}
1.8 & 1.1
\end{pmatrix}^\top$,
$v_{50} = 
\begin{pmatrix}
-1.6 & 1.4
\end{pmatrix}^\top$,
$v_{60} = 
\begin{pmatrix}
1.1 & -1.2
\end{pmatrix}^\top$.
In Figures \ref{fig_output1} and \ref{fig_output2}, we have plotted the trajectories of the output vectors $z_i$, $i =1,2\ldots, 6$ of the controlled network. The proposed protocol indeed achieves $\textbf{z}$-output synchronization for the network.
\begin{figure}[t]
	\centering
	\includegraphics[width=0.9\columnwidth]{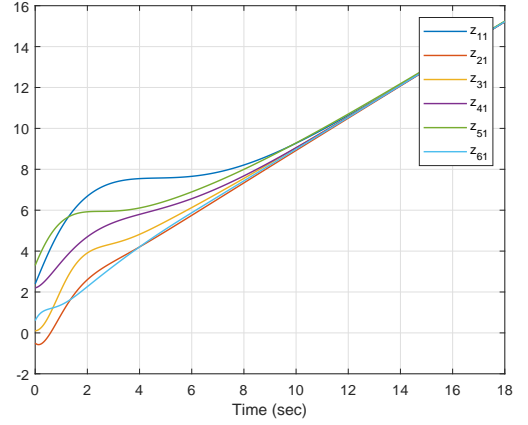}
	\caption{Plots of trajectories of the first component of the output vectors $z_{1},z_2, \ldots,z_6$} \label{fig_output1}
\end{figure}
\begin{figure}[t]
	\centering
	\includegraphics[width=0.9\columnwidth]{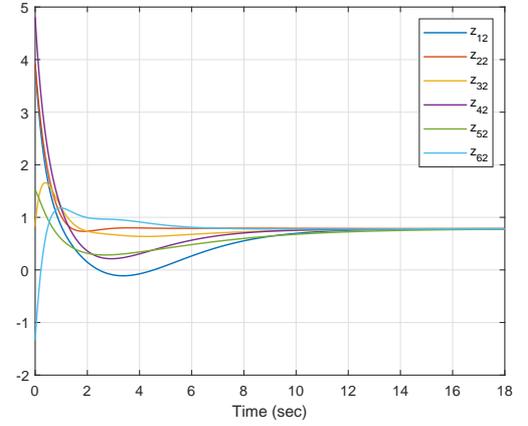}
	\caption{Plots of trajectories of the second component of the output vectors $z_{1}, z_2, \ldots,z_6$} \label{fig_output2}
\end{figure}
Moreover, for $i = 1,2 ,\ldots,6$, we compute  \[S_i = 	\textnormal{tr} (C_{1i} Q_i P_i Q_i C_{1i}^\top) +	\textnormal{tr} (C_{2i} Q_i C_{2i}^\top), \]
 and obtain that
\begin{align*}
S_1 = S_4 &=    0.6621, \ S_2 =S_5 =   0.4379,\
S_3 =S_6 =      0.3637.
\end{align*}
Note that, for all $i = 1,2 ,\ldots,6$, we have 
\[
	S_i < \frac{\gamma}{N\lambda_N} = 0.75,
\]
it then follows from Theorem \ref{main_thm} that the designed protocol is suboptimal, i.e. the associated cost is indeed smaller than the desired tolerance $\gamma = 18$.

\section{Conclusion}\label{sec_conclusion}
In this paper, we have studied the $\mathcal{H}_2$ suboptimal output synchronization problem for heterogeneous linear multi-agent systems. 
Given a heterogeneous multi-agent system and an associated $\mathcal{H}_2$  cost functional, we have provided a design method for computing dynamic output feedback based protocols that guarantee the associated cost to be smaller than a given upper bound while the controlled network achieves output synchronization. 
For each agent, its two local control gains are given in terms of solutions of two Riccati inequalities, each of dimension equal to that of the agent dynamics.
The computation of the local control gains  involves the largest eigenvalue of the Laplacian matrix of the communication graph.

\bibliographystyle{IEEEtran}
\bibliography{h2_auto}

\end{document}